\documentclass[12pt]{amsart}
\usepackage{amsthm}
\usepackage{amssymb}
\usepackage{geometry}
\usepackage{enumerate}
\usepackage{setspace}
\usepackage{upgreek}
\usepackage{pgfkeys}
\usepackage{tikz}
\usepackage{tikz-cd}
\usetikzlibrary{matrix, arrows}
\usepackage[unicode=true]
 {hyperref}
\hypersetup{
colorlinks=true,
urlcolor=black,
citecolor=blue,
linkcolor=blue,
}

\newtheorem{thm}{Theorem}[section]
\newtheorem{prop}[thm]{Proposition}

\theoremstyle{definition}
\newtheorem{definition}[thm]{Definition}

\theoremstyle{remark}
\newtheorem{remark}[thm]{Remark}

\numberwithin{equation}{section}

\begin{document}

\large

\title[Realizable Classes in the Unitary Class Group]{On the Realizable Classes of the Square Root of the Inverse Different in the Unitary Class Group}

\author{Cindy Tsang}
\address{Department of Mathematics, University of California, Santa Barbara}
\email{cindytsy@math.ucsb.edu}
\urladdr{http://math.ucsb.edu/$\sim$cindytsy} 

\date{November 29, 2015}

\begin{abstract}Let $K$ be a number field with ring of integers $\mathcal{O}_K$ and let $G$ be a finite abelian group of odd order. Given a $G$-Galois $K$-algebra $K_h$, let $A_h$ denote its square root of the inverse different, which exists by Hilbert's formula. If $K_h/K$ is weakly ramified, then the pair $(A_h,Tr_h)$ is locally $G$-isometric to $(\mathcal{O}_KG,t_K)$ and hence defines a class in the unitary class group $\mbox{UCl}(\mathcal{O}_KG)$ of $\mathcal{O}_KG$. Here $Tr_h$ denotes the trace of $K_h/K$ and $t_K$ the symmetric bilinear form on $\mathcal{O}_KG$ for which $t_K(s,t)=\delta_{st}$ for all $s,t\in G$. We study the collection of all such classes and show that a subset of them is in fact a subgroup of $\mbox{UCl}(\mathcal{O}_KG)$.
\end{abstract}

\maketitle

\tableofcontents

\section{Introduction}\label{s:1}

Let $K$ be a number field with ring of integers $\mathcal{O}_K$ and let $G$ be a finite group of odd order. The set of all isomorphism classes of $G$-Galois $K$-algebras (see Section~\ref{s:4} for a brief review of Galois algebras) is in one-one correspondence with the pointed set $H^1(\Omega_K,G)$, where $\Omega_K$ is the absolute Galois group of $K$ acting trivially on $G$. Given $h\in H^1(\Omega_K,G)$, we will write $K_h$ for a Galois algebra representative and $Tr_h$ for the trace of $K_h/K$. Moreover, we will write $A_h$ for the square root of the inverse different of $K_h/K$. Since $G$ has odd order,  the inverse different of $K_h/K$ indeed has a square root by Hilbert's formula (see \cite[Chapter IV, Proposition 4]{Serre}, for example).

If $K_h/K$ is weakly ramified (see Definition~\ref{ramification}), then it follows from \cite[Theorem 1 in Section 2]{Erez} that $A_h$ is locally free as an $\mathcal{O}_KG$-module and hence defines a class $\mbox{cl}(A_h)$ in the locally free class group $\mbox{Cl}(\mathcal{O}_KG)$ of $\mathcal{O}_KG$. Such a class in $\mbox{Cl}(\mathcal{O}_KG)$ is said to be \emph{$A$-realizable}, and \emph{tame $A$-realizable} if $K_h/K$ is tame. For $G$ abelian, in \cite{T} the author has studied these $A$-realizable classes using techniques developed by L. McCulloh in \cite{McCulloh}.  The purpose of this paper is to show that the same methods used in \cite{T} can be applied to study the structure of $A_h$ as an $\mathcal{O}_KG$-module equipped with the bilinear form induced by $Tr_h$. We will explain this in more detail below.

Given $h\in H^1(\Omega_K,G)$, it is well-known that $A_h$ is self-dual respect to $Tr_h$ (this follows from \cite[Chapter 3, (2.14)]{FT}, for example). \mbox{In other words, we have}
\[
A_h=\{a\in K_h\mid Tr_h(aA_h)\subset\mathcal{O}_K\}.
\]
In particular, the map $Tr_h$ induces a $G$-invariant symmetric $\mathcal{O}_K$-bilinear form
\[
A_h\times A_h\longrightarrow\mathcal{O}_K;\hspace{1cm}(a,b)\mapsto Tr_h(ab)
\]
on $A_h$. The pair $(A_h,Tr_h)$ is thus a \mbox{$G$-form over $\mathcal{O}_K$} (see Subsection~\ref{s:3.1} for a brief review of $G$-forms). On the other hand, there is a canonical symmetric $\mathcal{O}_K$-bilinear form $t_K$ on $\mathcal{O}_KG$ for which $t_K(s,t)=\delta_{st}$ for all $s,t\in G$. 

If $K_h/K$ is weakly ramified, then again $A_h$ is locally free over $\mathcal{O}_KG$ by \cite[Theorem 1 in Section 2]{Erez}. In other words, for each prime $v$ in $\mathcal{O}_K$, there exists an $\mathcal{O}_{K_v}G$-isomorphism $\mathcal{O}_{K_v}\otimes_{\mathcal{O}_K}A_h\simeq\mathcal{O}_{K_v}G$, where $\mathcal{O}_{K_v}$ denotes the ring of integers in the completion of $K$ with respect to $v$. It is natural to ask whether such an isomorphism may be chosen such that the bilinear forms $Tr_h$ and $t_K$ are preserved. More precisely, are the $G$-forms $(A_h,Tr_h)$ and $(\mathcal{O}_KG,t_K)$ locally $G$-isometric over $\mathcal{O}_K$ (see Definition~\ref{Giso})? For $G$ abelian and of odd order, the answer turns out to be affirmative (see Subsection~\ref{s:5.0}). In this case, the pair $(A_h,Tr_h)$ defines a class in the \emph{unitary class group UCl$(\mathcal{O}_KG)$ of $\mathcal{O}_KG$} (see Subsection~\ref{s:3.2}). By abuse of terminology, such a class in $\mbox{UCl}(\mathcal{O}_KG)$ is also said to be \emph{$A$-realizable}, and \emph{tame $A$-reazliable} if $K_h/K$ is tame.

In what follows, we assume that $G$ is abelian and of odd order. The pointed set $H^1(\Omega_K,G)$ is then equal to $\mbox{Hom}(\Omega_K,G)$ and hence has a group structure. Define
\[
H^1_t(\Omega_K,G):=\{h\in H^1(\Omega_K,G)\mid K_h/K\mbox{ is tame}\},
\]
which is a subgroup of $H^1(\Omega_K,G)$ (see Remark~\ref{tamesubgp}), and
\[
H^1_w(\Omega_K,G):=\{h\in H^1(\Omega_K,G)\mid K_h/K\mbox{ is weakly ramified}\}.
\]
Moreover, we will write
\[
\mathcal{A}_u(\mathcal{O}_KG):=\{\mbox{ucl}(A_h):h\in H_w^1(\Omega_K,G)\}
\]
for the set of all $A$-realizable classes in $\mbox{UCl}(\mathcal{O}_KG)$, and
\[
\mathcal{A}_u^t(\mathcal{O}_KG):=\{\mbox{ucl}(A_h):h\in H_t^1(\Omega_K,G)\}
\]
for the subset of $\mathcal{A}_u(\mathcal{O}_KG)$ consisting of the tame $A$-realizable classes. It has been shown in \cite[Theorem 4.1]{ErezMorales} that $\mathcal{A}_u(\mathbb{Z}G)=1$, and in \cite[Theorem 3.6]{Morales} that $\mathcal{A}_u^t(\mathcal{O}_KG)$ is subgroup of $\mbox{UCl}(\mathcal{O}_KG)$ when $G$ has odd prime order and $K$ contains all $|G|$-th roots of unity.

Now, consider the map
\[
\mbox{gal}_{A,u}:H^1_w(\Omega_K,G)\longrightarrow\mbox{UCl}(\mathcal{O}_KG);\hspace{1em}\mbox{gal}_{A,u}(h):=\mbox{ucl}(A_h).
\]
Note that $\mbox{gal}_{A,u}(H^1_w(\Omega_K,G))=\mathcal{A}_u(\mathcal{O}_KG)$ and $\mbox{gal}_{A,u}(H^1_t(\Omega_K,G))=\mathcal{A}^t_u(\mathcal{O}_KG)$ by definition. But $H^1_w(\Omega_K,G)$ is only a subset of $H^1(\Omega_K,G)$, and the map $\mbox{gal}_{A,u}$ is not a homomorphism in general even when restricted to the subgroup $H^1_t(\Omega_K,G)$. Hence, it is unclear whether $\mathcal{A}_u(\mathcal{O}_KG)$ and $\mathcal{A}^t_u(\mathcal{O}_KG)$ are subgroups of $\mbox{UCl}(\mathcal{O}_KG)$. Nevertheless, we will show that $\mbox{gal}_{A,u}$ preserves inverses and is \emph{weakly multiplicative}. More precisely, let $M_K$ be the set of primes in $\mathcal{O}_K$ and define 
\[
d(h):=\{v\in M_K\mid K_h/K\mbox{ is ramified at $v$}\}
\]
for $h\in H^1(\Omega_K,G)$. Analogous to \cite[Theorem 1.2]{T}, we will prove:

\begin{thm}\label{thm:weaku}Let $K$ be a number field and let $G$ be a finite abelian \mbox{group of} odd order. For all $h,h_1,h_2\in H^1_w(\Omega_K,G)$ with $d(h_1)\cap d(h_2)=\emptyset$, we have 
\begin{enumerate}[(a)]
\item  $h^{-1}\in H_w^1(\Omega_K,G)$ and $\mbox{gal}_{A,u}(h^{-1})=\mbox{gal}_{A,u}(h)^{-1}$;
\item $h_1h_2\in H^1_w(\Omega_K,G)$ and $
\mbox{gal}_{A,u}(h_1h_2)=\mbox{gal}_{A,u}(h_1)\mbox{gal}_{A,u}(h_2)$.
\end{enumerate}
\end{thm}

Analogous to \cite[Theorems 1.3 and 1.6]{T}, we will also give a characterization of the set $\mathcal{A}^t_u(\mathcal{O}_KG)$ (see (\ref{tamecharu})) and then prove:

\begin{thm}\label{thm:subgroupu}
Let $K$ be a number field and let $G$ be a finite abelian \mbox{group of} odd order. Then, the set $\mathcal{A}_u^t(\mathcal{O}_KG)$ is a subgroup of $\mbox{UCl}(\mathcal{O}_KG)$.  Moreover, given $c\in \mathcal{A}_u^t(\mathcal{O}_KG)$ and a finite set $T$ of primes in $\mathcal{O}_K$, there exists $h\in H^1_t(\Omega_K,G)$ such that
\begin{enumerate}[(1)]
\item $K_h/K$ is a field extension;
\item $K_h/K$ is unramified at all $v\in T$;
\item $c=\mbox{ucl}(A_h)$.
\end{enumerate}
\end{thm}

\begin{thm}\label{thm:wildu}Let $K$ be a number field and let $G$ be a finite abelian \mbox{group of} odd order. Let $h\in H^1_w(\Omega_K,G)$ and let $V$ denote the set of primes in $\mathcal{O}_K$ which are wildly ramified in $K_h/K$. If
\begin{enumerate}[(1)]
\item every $v\in V$ is unramified over $\mathbb{Q}$; and
\item the ramification index of every $v\in V$ in $K_h/K$ is prime,
\end{enumerate}
then we have $\mbox{ucl}(A_h)\in\mathcal{A}_u^t(\mathcal{O}_KG)$. 
\end{thm}

\begin{remark}Notice that there is a natural homomorphism
\[
\Phi: \mbox{UCl}(\mathcal{O}_KG)\longrightarrow\mbox{Cl}(\mathcal{O}_KG);\hspace{1em}\mbox{ucl}((X,T))\mapsto \mbox{cl}(X)
\]
afforded by forgetting the given $G$-invariant symmetric $\mathcal{O}_K$-bilinear form $T$ on $X$ for any locally free $\mathcal{O}_KG$-module $X$. 
Theorems~\ref{thm:weaku}, ~\ref{thm:subgroupu}, and~\ref{thm:wildu} are therefore refinements of \cite[Theorems 1.2, 1.3, and 1.6]{T}, respectively. In fact, their proofs are essentially the same.
\end{remark}

Here is a brief outline of the contents of the rest of this paper.  In Section~\ref{s:3}, we will define the unitary class group, which was first introduced by J. Morales in \cite{Morales}. In Section~\ref{s:4}, we will give a brief review of Galois algebras and resolvends. Then, in Section~\ref{s:5}, we will recall the necessary definitions to give a characterization of the set $\mathcal{A}_u^t(\mathcal{O}_KG)$ (see (\ref{tamecharu})). \mbox{We will} prove Theorem~\ref{thm:weaku} in Subsection~\ref{s:5.1}, and then Theorems~\ref{thm:subgroupu} and ~\ref{thm:wildu} in Subsection~\ref{s:5.5}. To avoid repetition, we will use certain results from \cite{T} \mbox{without restating them.}

\section{Notation and Conventions}\label{s:2}

Throughout this paper, we will fix a number field $K$ and a finite group $G$. We will also use the convention that the homomorphisms in the cohomology groups considered are all continuous.

The symbol $F$ will denote a number field or a finite extension of $\mathbb{Q}_p$, where $p$ is a prime number. Given such an $F$, we will define:
\begin{align*}
\mathcal{O}_F&:=\mbox{the ring of integers in $F$};\\
F^c&:=\mbox{a fixed algebraic closure of $F$};\\
\Omega_F&:=\mathit{Gal}(F^c/F);\\
F^t&:=\mbox{the maximal tamely ramified extension of $F$ in $F^c$};\\
\Omega^t_F&:=\mathit{Gal}(K^t/K);\\
M_F&:=\mbox{the set of all finite primes in $F$};\\
[-1]&:=\mbox{the involution on $F^cG$ induced by the involution $s\mapsto s^{-1}$ on $G$};\\
t_F&:=\mbox{the symmetric $\mathcal{O}_F$-bilinear form $\mathcal{O}_FG\times\mathcal{O}_FG\longrightarrow\mathcal{O}_F$}\\
&\hspace{0.65cm}\mbox{for which $t_F(s,t)=\delta_{st}$ for all $s,t\in G$}.
\end{align*}
For $G$ abelian, we will further define:
\[
\widehat{G}:=\mbox{the group of irreducible $F^c$-valued characters on $G$}.
\]
We will also let $\Omega_F$ and $\Omega_F^t$ act trivially on $G$ from the left. Moreover, we will choose a compatible set $\{\zeta_n:n\in\mathbb{Z}^+\}$ of primitive roots of unity in $F^c$. 

For $F$ a number field and $v\in M_F$, we adopt the following notation:
\begin{align*}
F_v&:=\mbox{the completion of $F$ with respect to $v$};\\
i_v&:=\mbox{a fixed embedding $F^c\longrightarrow F_v^c$ extending the natural}\\
&\hspace{0.65cm}\mbox{embedding $F\longrightarrow F_v$};\\
\widetilde{i_v}&:=\mbox{the embedding $\Omega_{F_v}\longrightarrow\Omega_{F}$ induced by $i_v$}.
\end{align*}
Moreover, if $\{\zeta_n:n\in\mathbb{Z}^+\}$ is the chosen compatible set of primitive roots of unity in $F^c$, then for each $v\in M_F$ we choose $\{i_v(\zeta_n):n\in\mathbb{Z}^+\}$ to be the compatible set of primitive roots of unity in $F^c_v$.

\section{$G$-Forms and Unitary Class Groups}\label{s:3}

\subsection{$G$-Forms}\label{s:3.1} Let $F$ be a number field or a finite extension of $\mathbb{Q}_p$. We will give a brief review of $G$-forms over $\mathcal{O}_F$ and their basic properties.

\begin{definition}An \emph{$\mathcal{O}_FG$-lattice} is a left $\mathcal{O}_FG$-module which is finitely generated and projective as an $\mathcal{O}_F$-module. 
\end{definition}

\begin{definition}\label{Gform}
A \emph{$G$-form over $\mathcal{O}_F$} is pair $(X,T)$, where $X$ is an $\mathcal{O}_FG$-lattice and $T:X\times X\longrightarrow \mathcal{O}_F$ is $G$-invariant symmetric $\mathcal{O}_F$-bilinear form on \mbox{$X$. Two} $G$-forms $(X,T)$ and $(X',T')$ over $\mathcal{O}_F$ are said to be \emph{$G$-isometric over $\mathcal{O}_F$} if there exists an isomorphism $\varphi:X\longrightarrow X'$ of $\mathcal{O}_FG$-modules such that
\[
T'(\varphi(x),\varphi(y))=T(x,y)\hspace{1cm}\mbox{for all }x,y\in X.
\]
Such an isomorphism is called a \emph{$G$-isometry over $\mathcal{O}_F$}. The $G$-isometry class of $(X,T)$ will be denoted by $\mbox{ucl}((X,T))$.
\end{definition}

Given a $G$-form $(X,T)$ over $\mathcal{O}_F$, notice that the map $T$ extends uniquely to a $G$-invariant symmetric $F$-bilinear on $F\otimes_{\mathcal{O}_F} X$ via linearity. By abuse of notation, we will still use $T$ to denote this bilinear form.

\begin{definition}
Let $(X,T)$ be a $G$-form over $\mathcal{O}_F$. The \emph{dual of $X$ with respect to $T$}, or simply the \emph{dual of $X$}, is defined to be the $\mathcal{O}_F$-module 
\[
X^*:=\{x\in F\otimes_{\mathcal{O}_{F}}X\mid T(x,X)\subset\mathcal{O}_F\}.
\]
The $G$-form $(X,T)$ is said to be \emph{self-dual with respect to $T$}, or simply \emph{self-dual} if $X=X^*$. Moreover, an element $x\in F\otimes_{\mathcal{O}_{F}}X$ is said to be \emph{self dual with respect to $T$}, or simply \emph{self-dual} if
\[
T(sx,tx)=\delta_{st}\hspace{1cm}\mbox{for all }s,t\in G.
\]
\end{definition}

Recall that $t_F$ denotes the canonical symmetric $\mathcal{O}_F$-bilinear form on $\mathcal{O}_FG$ for which $t_F(s,t)=\delta_{st}$ for all $s,t\in G$. The $G$-forms $(X,T)$ over $\mathcal{O}_F$ which are $G$-isometric to $(\mathcal{O}_FG,t_F)$ are precisely those for which $X$ has a free self-dual generator over $\mathcal{O}_FG$.

\begin{prop}\label{prop:Gform}
A $G$-form $(X,T)$ over $\mathcal{O}_F$ is $G$-isometric to $(\mathcal{O}_FG,t_F)$ if and \mbox{only if} there exists $x\in X$ such that $X=\mathcal{O}_{F}G\cdot x$ and $x$ is self-dual.
\end{prop}
\begin{proof}If $(X,T)$ is $G$-isometric to $(\mathcal{O}_FG,t_F)$ and $\varphi:\mathcal{O}_FG\longrightarrow X$ is a $G$-isometry over $\mathcal{O}_F$, then clearly $x:=\varphi(1)$ is self-dual and $X=\mathcal{O}_FG\cdot x$. Conversely, if $x\in X$ is self-dual and $X=\mathcal{O}_FG\cdot x$, then
\[
\varphi:\mathcal{O}_FG\longrightarrow X;\hspace{1em}\varphi(\gamma):=\gamma\cdot x
\]
is a $G$-isometry over $\mathcal{O}_F$. This proves the claim.
\end{proof}

Recall that $[-1]$ denotes the involution on $F^cG$ induced by the involution $s\mapsto s^{-1}$ on $G$. This involution may be used to identify elements $c\in (FG)^\times$ for which $\mathcal{O}_FG\cdot c$ is self-dual and those which are self-dual.

\begin{prop}\label{SDcriteria}
Let $c\in (FG)^\times$. 
\begin{enumerate}[(a)]
\item The $\mathcal{O}_FG$-lattice $\mathcal{O}_FG\cdot c$ is self-dual if and only if $cc^{[-1]}\in(\mathcal{O}_FG)^\times$.
\item The element $c$ is self-dual if and only if $cc^{[-1]}=1$.
\end{enumerate}
\end{prop}
\begin{proof}An element $\gamma\in FG$ lies in $\gamma \in(\mathcal{O}_FG\cdot c)^*$ if and only if
\begin{equation}\label{gamma}
t_F(\gamma,sc)\in \mathcal{O}_F\hspace{1cm}\mbox{for all }s\in G
\end{equation}
because $t_F$ is $\mathcal{O}_F$-bilinear. But for any $s\in G$, we have $t_F(\gamma,sc)=t_F(\gamma c^{[-1]},s)$, which is simply the coefficient of $s$ in $\gamma c^{[-1]}$. Hence, condition (\ref{gamma}) holds precisely when $\gamma c^{[-1]}\in\mathcal{O}_FG$. This means that
\[
(\mathcal{O}_FG\cdot c)^*=\mathcal{O}_FG\cdot (c^{[-1]})^{-1}
\]
and so $\mathcal{O}_FG\cdot c$ is self-dual if and only if $cc^{[-1]}\in(\mathcal{O}_FG)^\times$. This proves (a). 

For (b), simply observe that
\[
t_F(sc,tc)=t_F(s\cdot cc^{[-1]},t)\hspace{1cm}\mbox{for all }s,t\in G.
\]
It follows that $c$ is self-dual if and only if $cc^{[-1]}=1$.
\end{proof}

\begin{definition}
In view of Proposition~\ref{SDcriteria}, define
\begin{align*}
FG_{(s)}&:=\{c\in (FG)^\times\mid cc^{[-1]}\in (\mathcal{O}_FG)^\times\};\\
FG_{(1)}&:=\{c\in (FG)^\times\mid cc^{[-1]}=1\}.
\end{align*}
\end{definition}

Now, assume that $F$ is a number field. Given a $G$-form $(X,T)$ over $\mathcal{O}_F$, notice that the map $T$ extends uniquely to a $G$-invariant symmetric $\mathcal{O}_{F_v}$-bilinear on $X_v:=\mathcal{O}_{F_v}\otimes_{\mathcal{O}_F} X$ via linearity for each $v\in M_F$. We will denote these bilinear forms by $T_v$.

\begin{definition}\label{Giso}
A $G$-form $(X,T)$ over $\mathcal{O}_F$ is said to be \emph{locally $G$-isometric to $(\mathcal{O}_FG,t_F)$} if $(X_v,T_v)$ and $(\mathcal{O}_{F_v}G,t_{F_v})$ are $G$-isometric over $\mathcal{O}_{F_v}$ for all $v\in M_F$. We will write  $g(\mathcal{O}_FG)_s$ for the set of all such $G$-forms $(X,T)$ over $\mathcal{O}_F$ that is $G$-isometric to $(\mathcal{O}_FG\cdot c,t_F)$ for some $c\in J(FG)$.
\end{definition}

\subsection{Unitary Class Groups}\label{s:3.2} Let $F$ be a number field. Moreover, assume that $G$ is abelian and of odd order. Following \cite[Section]{Morales}, we will define the unitary class group of $\mathcal{O}_FG$. Our approach will be slightly different, \mbox{but the} resulting group is canonically isomorphic to that defined in \cite{Morales}.

As a set, the unitary class group of $\mathcal{O}_FG$ is defined to be
\[
\mbox{UCl}(\mathcal{O}_FG):=\{\mbox{ucl}((X,T)):(X,T)\in g(\mathcal{O}_FG)_s\}.
\]
We will show that this set has a group structure by giving it an id\`{e}lic description. The key lies in the following theorem.

\begin{thm}\label{ErezMor} Let $X$ be an $\mathcal{O}_FG$-lattice in $FG$. We have $(X,t_F)\in g(\mathcal{O}_FG)_s$ if and only if $X$ is locally free over $\mathcal{O}_FG$ and self-dual with respect to $t_F$.
\end{thm}
\begin{proof}If $(X,t_F)\in g(\mathcal{O}_FG)_s$, then $X$ is locally free over $\mathcal{O}_FG$ and self-dual with respect to $t_F$ by Proposition~\ref{prop:Gform}. As for the converse, see \cite[Corollary 2.4]{ErezMorales}; we remark that this  requires that $G$ is abelian and of odd order.
\end{proof}

Next, we give an id\`{e}lic description of the set $\mbox{UCl}(\mathcal{O}_KG)$.

\begin{definition}
Let $J(FG)$ and $J(FG_{(s)})$ be the restricted direct \mbox{products of} the groups $(F_vG)^\times$ and $F_vG_{(s)}$, respectively, with respect to the subgroups $(\mathcal{O}_{F_v}G)^\times$ for $v\in M_F$. Moreover, let
\[
\partial:(FG)^\times\longrightarrow J(FG)
\]
be the diagonal map and let
\[
U(\mathcal{O}_FG):=\prod_{v\in M_F}(\mathcal{O}_{F_v}G)^\times
\]
be the group of unit id\`{e}les. 
\end{definition}

For each id\`{e}le $c\in J(FG)$, define
\[
\mathcal{O}_FG\cdot c:=\Bigg(\bigcap_{v\in M_F}\mathcal{O}_{F_v}G\cdot c_v\Bigg)\cap FG.
\]
It is well-known that the locally free $\mathcal{O}_FG$-lattices in $FG$ are all of this form.

\begin{prop}\label{inFG}
Let $c,c'\in J(FG)$. 
\begin{enumerate}[(a)]
\item The $G$-form $(\mathcal{O}_FG\cdot c,t_F)$ belongs to $g(\mathcal{O}_FG)_s$ if and only if $c\in J(FG_{(s)})$.
\item The $G$-forms $(\mathcal{O}_FG\cdot c,t_F)$ and $(\mathcal{O}_FG\cdot c',t_F)$ are $G$-isometric over $\mathcal{O}_F$ if and only if
\[
c'c^{-1}\in\partial(FG_{(1)})U(\mathcal{O}_{F}G)
\]
\end{enumerate}
\end{prop}
\begin{proof}For (a), it follows directly from Proposition~\ref{SDcriteria} (a) and Theorem~\ref{ErezMor}. As for (b), observe that an isomorphism $\mathcal{O}_FG\cdot c\longrightarrow \mathcal{O}_FG\cdot c'$ is of the form
\[
\varphi:\mathcal{O}_FG\cdot c\longrightarrow \mathcal{O}_FG\cdot c';\hspace{1em}\varphi(x):=\gamma\cdot x,
\]
where $\gamma\in (FG)^\times$. Moreover, the map $\varphi$ is well-defined and is an isomorphism of $\mathcal{O}_FG$-modules if and only if
\begin{equation}\label{c'c}
c'c^{-1}\in\partial(\gamma)\cdot U(\mathcal{O}_FG).
\end{equation}
Morever, since $t_F$ is $\mathcal{O}_F$-bilinear and $G$-invariant, the map $\varphi$ is a $G$-isometry over $\mathcal{O}_F$ if and only if
\[
t_F(\gamma\cdot c,\gamma\cdot sc)=t_F(c,sc)\hspace{1cm}\mbox{for all }s\in G.
\]
But the above simplifies to
\[
t_F(\gamma\gamma^{[-1]}cc^{[-1]},s)=t_F(cc^{[-1]},s)\hspace{1cm}\mbox{for all }s\in G,
\]
which holds precisely when $\gamma\gamma^{[-1]}cc^{[-1]}=cc^{[-1]}$, or equivalently $\gamma\gamma^{[-1]}=1$. This shows that $\varphi$ is a $G$-isometry over $\mathcal{O}_F$ if and only if $\gamma\in FG_{(1)}$, and the claim now follows from (\ref{c'c}).
\end{proof}

Now, consider the map
\[
j_{(s)}:J(FG_{(s)})\longrightarrow\mbox{UCl}(\mathcal{O}_FG);
\hspace{1em}j_{(s)}(c):=\mbox{ucl}((\mathcal{O}_FG\cdot c,t_F)).
\]
By Proposition~\ref{inFG}, the map $j_{(s)}$ is well-defined and it induces an injection
\begin{equation}\label{isoUCG}
\frac{J(FG_{(s)})}{\partial(FG_{(1)})U(\mathcal{O}_{F}G)}\longrightarrow\mbox{UCl}(\mathcal{O}_FG).
\end{equation}
By definition of $g(\mathcal{O}_FG)_s$, the above is also a surjection and hence a bijection. Since the quotient on the left is a group, this bijection induces a group structure on $\mbox{UCl}(\mathcal{O}_FG)$.

\begin{definition}\label{UCG}
The \emph{unitary class group of $\mathcal{O}_FG$} is defined to be the set
\[
\mbox{UCl}(\mathcal{O}_FG):=\{\mbox{ucl}((X,T)):(X,T)\in g(\mathcal{O}_FG)_s\}
\]
equipped with the group structure induced by the bijection (\ref{isoUCG}).
\end{definition}

\section{Galois Algebras and Resolvends}\label{s:4}

Let $F$ be a number field or a finite extension of $\mathbb{Q}_p$. The group $G$ is arbitrary for the moment, but will be soon assumed to be abelian. We will give a brief review of Galois algebras and resolvends (see \cite[Section 1]{McCulloh} for more details).

\begin{definition}\label{GaloisAlg}A \emph{Galois algebra over $F$ with group $G$} or \emph{$G$-Galois $F$-algebra} is a commutative semi-simple $F$-algebra $N$ on which $G$ acts on the left as a group of automorphisms such that $N^{G}=F$ and $[N:F]=|G|$. Two $G$-Galois $F$-algebras are \emph{isomorphic} if there is an $F$-algebra isomorphism between them which preserves the action of $G$.
\end{definition}

Recall that $\Omega_F$ acts trivially on $G$. Then, the set of isomorphism classes of $G$-Galois $F$-algebras is in one-one correspondence with the pointed set
\begin{equation}\label{ptset}
H^1(\Omega_F,G):=\mbox{Hom}(\Omega_F,G)/\mbox{Inn}(G).
\end{equation}
In particular, each $h\in\mbox{Hom}(\Omega_F,G)$ is associated to the $F$-algebra
\[
F_{h}:=\mbox{Map}_{\Omega_F}(^{h}G,F^{c}),
\]
where $^{h}G$ is the group $G$ endowed with the $\Omega_F$-action given by
\[
(\omega\cdot s):=h(\omega)s\hspace{1cm}\mbox{for $s\in G$ and $\omega\in\Omega_F$}.
\]
The $G$-action on $F_{h}$ is defined by
\[
(s\cdot a)(t):=a(ts)\hspace{1cm}\mbox{for $a\in F_h$ and $s,t\in G$}.
\]
Now, choose a set $\{s_i\}$ of coset representatives for $h(\Omega_F)\backslash G$, then each $a\in F$ is determined by the values $a(s_i)$, and clearly each $a(s_i)$ may be arbitrarily chosen provided that it is fixed by all $\omega\in\ker(h)$. Hence, if
\begin{equation}\label{Fh}
F^{h}:=(F^{c})^{\ker(h)},
\end{equation}
then the choices of the coset representatives $\{s_i\}$ induce an isomorphism
\[
F_{h}\simeq \prod_{h(\Omega_F)\backslash G}F^{h}
\]
of $F$-algebras. Since $h$ induces an isomorphism $\mathit{Gal}(F^h/F)\simeq h(\Omega_F)$, we have
\[
[F_h:F]=[G:h(\Omega_F)][F^h:F]=|G|.
\]
Viewing $F$ as embedded in $F_h$ as the constant $F$-valued functions, we have $F_h^G=F$ as well. Hence, indeed $F_h$ is a $G$-Galois $F$-algebra.

It is not difficult to verify that every $G$-Galois $F$-algebra is isomorphic to some $F_h$ arising from a homomorphism $h\in\mbox{Hom}(\Omega_F,G)$, and that for $h,h'\in\mbox{Hom}(\Omega_F,G)$ we have $F_{h}\simeq F_{h'}$ if and only if $h$ and $h'$ differ by an element in $\mbox{Inn}(G)$. Hence, indeed the set of isomorphism classes of $G$-Galois $F$-algebras is in bijection with (\ref{ptset}).

In the rest of this section, we will assume that $G$ is abelian. In this case, the pointed set $H^1(\Omega_F,G)$ is equal to $\mbox{Hom}(\Omega_F,G)$ and in particular is a group.

\begin{definition}
Given $h\in\mbox{Hom}(\Omega_F,G)$, let $F^h$ be as in (\ref{Fh}). Let $\mathcal{O}^h:=\mathcal{O}_{F^h}$ and define the \emph{ring of integers of $F_h$} by
\[
\mathcal{O}_h:=\mbox{Map}_{\Omega_F}(^hG,\mathcal{O}^h).
\]
If the inverse different of $F^h/F$ has a square root, denote it by $A^h$ and define the \emph{square root of the inverse different of $F_h/F$} by
\[
A_h:=\mbox{Map}_{\Omega_F}(^hG,A^h).
\]
\end{definition}

\begin{remark}\label{localize}
For $F$ a number field and $h\in\mbox{Hom}(\Omega_F,G)$, define
\[
h_v\in\mbox{Hom}(\Omega_{F_v},G);\hspace{1em}h_v:=h\circ\widetilde{i_v}
\]
for each $v\in M_F$. It is proved in \cite[(1.4)]{McCulloh} that
\[
(F_v)_{h_v}\simeq F_v\otimes_FF_h.
\]
Consequently, we have
\begin{align*}
\mathcal{O}_{h_v}&\simeq\mathcal{O}_{F_v}\otimes_{\mathcal{O}_F}\mathcal{O}_h;\\
A_{h_v}&\simeq\mathcal{O}_{F_v}\otimes_{\mathcal{O}_F}A_h,
\end{align*}
where we implicitly assume that $A^h$ and $A^{h_v}$ exist in the second isomorphism.
\end{remark}

\begin{definition}\label{ramification}
Given $h\in\mbox{Hom}(\Omega_F,G)$, we say that $F_h/F$ or $h$ is \emph{unramified} if $F^h/F$ is unramified. Similarly for \emph{tame}, \emph{wild}, and \emph{weakly ramified}. Recall that a Galois extension over $F$ is said to be weakly ramified if all of the second ramification groups (in lower numbering) attached to it are trivial.
\end{definition}

\begin{remark}\label{tamesubgp}
A homomorphism $h\in\mbox{Hom}(\Omega_F,G)$ is tame if and only if it factors through the quotient map $\Omega_F\longrightarrow\Omega_F^t$. Hence, the subset of $\mbox{Hom}(\Omega_F,G)$ consisting of the tame homomorphisms may be naturally identified with $\mbox{Hom}(\Omega_F^t,G)$, and is in particular a subgroup of $\mbox{Hom}(\Omega_F,G)$.
\end{remark}

Next, consider the $F^c$-algebra $\mbox{Map}(G,F^c)$ on which we let $G$ act via
\[
(s\cdot a)(t):=a(ts)\hspace{1cm}\mbox{for $a\in\mbox{Map}(\Omega_F,G)$ and $s,t\in G$}.
\]
Note that $F_h$ is an $FG$-submodule of $\mbox{Map}(G,F^c)$ for all $h\in\mbox{Hom}(\Omega_F,G)$. 

\begin{definition}\label{resolvend}
The \emph{resolvend map} $\mathbf{r}_{G}:\mbox{Map}(G,F^{c})\longrightarrow F^{c}G$ is defined by
\[
\mathbf{r}_{G}(a):=\sum\limits _{s\in G}a(s)s^{-1}.
\]
\end{definition}

It is clear that $\mathbf{r}_{G}$ is an isomorphism of $F^cG$-modules, but not an isomorphism of $F^cG$-algebras because it does not preserve multiplication. Moreover, given $a\in\mbox{Map}(G,F^c)$, we have that $a\in F_h$ if and only if
\begin{equation}\label{resol1}
\omega\cdot\mathbf{r}_{G}(a)=\mathbf{r}_{G}(a)h(\omega)
\hspace{1cm}\mbox{for all }\omega\in\Omega_F.
\end{equation}
In particular, if $\mathbf{r}_{G}(a)$ is invertible, then $h$ is given by
\begin{equation}\label{resol2}
h(\omega)=\mathbf{r}_{G}(a)^{-1}(\omega\cdot\mathbf{r}_{G}(a))
\hspace{1cm}\mbox{for all $\omega\in\Omega_F$}.
\end{equation}
The next proposition shows that resolvends may be used to identify elements $a\in F_h$ for which $F_h=FG\cdot a$ or $\mathcal{O}_h=\mathcal{O}_FG\cdot a$.

\begin{prop}\label{NBG}Let $a\in F_h$.
\begin{enumerate}[(a)]
\item We have $F_h=FG\cdot a$ if and only if $\mathbf{r}_{G}(a)\in (F^{c}G)^{\times}$.
\item We have $\mathcal{O}_h=\mathcal{O}_FG\cdot a$ with $h$ unramified if and only if $\mathbf{r}_G(a)\in(\mathcal{O}_{F^c}G)^\times$. Furthermore, if $F$ is a finite extension of $\mathbb{Q}_p$ and $h$ is unramified, then there exists $a\in\mathcal{O}_h$ such that $\mathcal{O}_h=\mathcal{O}_FG\cdot a$.
\end{enumerate}
\end{prop}
\begin{proof}See \cite[Proposition 1.8]{McCulloh} for (a) and \cite[(2.11)]{McCulloh} for the first claim in (b). As for the second claim in (b), it follows from a classical theorem of Noether, or alternatively from \cite[Proposition 5.5]{McCulloh}.
\end{proof}

Now, let $Tr:\mbox{Map}(G,F^c)\longrightarrow F^cG$ be the standard algebra trace map defined by
\[
Tr(a):=\sum_{s\in G}a(s).
\]
This restricts to the trace $Tr_h: F_h\longrightarrow F$ of $F_h$ for each $h\in\mbox{Hom}(\Omega_F,G)$. By abuse of notation, we will also write $Tr_h$ for the $G$-invariant symmetric $F$-bilinear form $(a,b)\mapsto Tr_h(ab)$ on $F_h$ induced by $Tr_h$.

Resolvends may also be used to identify elements $a\in F_h$ for which $\mathcal{O}_FG\cdot a$ is a full self-dual $\mathcal{O}_FG$-lattice in $F_h$ and those which are self-dual with respect to $Tr_h$ (cf. Proposition~\ref{SDcriteria}). Notice that $\mathcal{O}_FG\cdot a$ is a full $\mathcal{O}_FG$-lattice in $F_h$ if and only if $\mathbf{r}_G(a)\in (F^cG)^\times$ by Proposition~\ref{NBG} (a).

\begin{prop}\label{SDcriteria'}Let $a\in F_h$ be such that $\mathbf{r}_G(a)\in (F^cG)^\times$.
\begin{enumerate}[(a)]
\item
The $\mathcal{O}_FG$-lattice $\mathcal{O}_FG\cdot a$ is self-dual if and only if $\mathbf{r}_G(a)\mathbf{r}_G(a)^{[-1]}\in(\mathcal{O}_FG)^\times$.
\item The element $a$ is self-dual if and only if $\mathbf{r}_G(a)\mathbf{r}_G(a)^{[-1]}=1$.
\end{enumerate}
\end{prop}
\begin{proof}See \cite[Proposition 2.8]{T} for (a). As for (b), it follows directly from the simple calculation that
\[
\mathbf{r}_G(a)\mathbf{r}_G(b)^{[-1]}=\sum_{s\in G}Tr((s\cdot a)b)s^{-1}\in FG
\]
for all $a,b\in F_h$.
\end{proof}

\section{The Class of the Square Root of the Inverse Different}\label{s:5}

\subsection{Computation using Resolvends}\label{s:5.0}

Let $F$ be a number field. Moreover, assume that $G$ is \mbox{abelian and of odd} order. Below, we explain why $(A_h,Tr_h)$ is locally $G$-isometric to $(\mathcal{O}_FG,t_F)$ for $h\in\mbox{Hom}(\Omega_F,G)$ weakly ramified and how the class $\mbox{ucl}(A_h)$ it defines in $\mbox{UCl}(\mathcal{O}_FG)$ may be computed using resolvends.

Let $h\in\mbox{Hom}(\Omega_F,G)$ be weakly ramified. Recall that $A_h$ is locally free over $\mathcal{O}_FG$ by \cite[Theorem 1 in Section 2]{Erez} in this case, and that $\mathcal{O}_{F_v}\otimes_{\mathcal{O}_F}A_h\simeq A_{h_v}$ from Remark~\ref{localize}. Hence, for each $v\in M_F$, there exists $a_v\in A_{h_v}$ such that
\begin{equation}\label{av}
A_{h_v}=\mathcal{O}_{F_v}G\cdot a_v.
\end{equation}
Moreover, by the Normal Basis Theorem, there exists $b\in F_h$ such that
\begin{equation}\label{b}
F_h=FG\cdot b.
\end{equation}
Since $G$ has odd order, it follows from \cite[Proposition 5.1]{BayerLenstra} that $b\in F_h$ may be chosen to be self-dual. Notice that $F_vG\cdot a_v=F_{h_v}=F_vG\cdot b$ for all $v\in M_F$ and that $\mathcal{O}_{F_v}G\cdot a_v=\mathcal{O}_{F_v}G\cdot b$ for all but finitely may $v\in M_F$. This implies that there exists $c\in J(FG)$ such that
\begin{equation}\label{cv}
a_v=c_v\cdot b
\end{equation}
for $v\in M_F$. In particular, the isomorphism
\[
FG\longrightarrow F_h;\hspace{1em}\gamma\mapsto \gamma\cdot b
\]
of $FG$-modules restricts to an isomorphism $\varphi:\mathcal{O}_FG\cdot c\longrightarrow A_h$ of $\mathcal{O}_FG$-modules. Since $b$ is chosen to be self-dual, the map $\varphi$ is in fact a $G$-isometry over $\mathcal{O}_F$. Moreover, 
the lattice $\mathcal{O}_FG\cdot c$ is self-dual with respect to $t_F$ because $A_h$ is self-dual with respect to $Tr_h$. It then follows from Proposition~\ref{SDcriteria} (a) that $c\in J(FG_{(s)})$, and from Proposition~\ref{inFG} (a) that $(A_h,Tr_h)\in g(\mathcal{O}_FG)_s$. In particular, we have $\mbox{ucl}(A_h)=\mbox{ucl}((\mathcal{O}_FG\cdot c,t_F))=j_{(s)}(c)$. Recall also that the resolvend map $\mathbf{r}_G:\mbox{Map}(G,F_v^c)\longrightarrow F_v^cG$ is an isomorphism of $F_v^cG$-modules for each $v\in M_F$. Thus, equation (\ref{cv}) is equivalent to
\begin{equation}\label{resolvendeq}
\mathbf{r}_G(a_v)=c_v\cdot\mathbf{r}_G(b).
\end{equation}
With this observation, we are now ready to prove Theorem~\ref{thm:weaku}.

\subsection{Proof of Theorem~\ref{thm:weaku}}\label{s:5.1}

\begin{proof}[Proof of Theorem~\ref{thm:weaku}]To prove (a), let $h\in H^1_w(\Omega_K,G)$. The fact that $h^{-1}\in H^1_w(\Omega_K,G)$ is a direct consequence of \cite[Proposition 5.1 (a)]{T}.

Let $b\in K_h$ be as in (\ref{b}), where we choose $b$ to be self-dual. Moreover, for each $v\in M_K$, let $a_v\in A_{h_v}$ and $c_v\in (K_vG)^\times$ be as in (\ref{av}) and (\ref{cv}), respectively. We have $c:=(c_v)\in J(KG_{(s)})$ and $\mbox{ucl}(A_h)=j_{(s)}(c)$, as explained in Subsection~\ref{s:5.0}.

Now, note that $\mathbf{r}_G(b)\in (K^cG)^\times$ by Proposition~\ref{NBG} (a). Because $\mathbf{r}_G$ is bijective, there exists $b'\in\mbox{Map}(G,K^c)$ such that
\[
\mathbf{r}_G(b')=\mathbf{r}_G(b)^{-1}.
\]
By (\ref{resol1}) and Proposition~\ref{NBG} (a), in fact $b'\in K_{h^{-1}}$ and $K_{h^{-1}}=KG\cdot b'$. Moreover, by Proposition~\ref{SDcriteria'} (b), clearly $b'$ is self-dual with respect to $Tr_{h^{-1}}$. Next, for each $v\in M_K$, there exists $a_v'\in A_{h_v^{-1}}$ such that $A_{h_v^{-1}}=\mathcal{O}_{K_v}G\cdot a_v'$ and
\[
\mathbf{r}_G(a_v')=\mathbf{r}_G(a_v)^{-1}
\]
by \cite[Proposition 5.1 (b)]{T}. Then, for each $v\in M_K$, equation (\ref{resolvendeq}) implies that
\[
\mathbf{r}_G(a_v')=c_v^{-1}\cdot\mathbf{r}_G(b')
\]
and so $a_v'=c_v^{-1}\cdot b'$. As in Subsection~\ref{s:5.0}, this implies that $\mbox{ucl}(A_{h^{-1}})=j_{(s)}(c^{-1})=\mbox{ucl}(A_h)^{-1}$, which proves (a).

Next, to prove (b), let $h_1,h_2\in H^1_w(\Omega_K,G)$ be such that $d(h_1)\cap d(h_2)=\emptyset$. The fact that $h_1h_2\in H^1_w(\Omega_K,G)$ is a direct consequence of \cite[Proposition 5.3 (b)]{T}.

For $i\in\{1,2\}$, let $b_i\in K_{h_i}$ be as in (\ref{b}), where we choose $b_i$ to be self-dual. Moreover, for each $v\in M_K$, let $a_{i,v}\in A_{(h_i)_v}$ and $c_{i,v}\in (K_vG)^\times$ be as in (\ref{av}) and (\ref{cv}), respectively. As explained in Subsection~\ref{s:5.0}, we have $c_i:=(c_{i,v})\in J(KG_{(s)})$ and $j_{(s)}(c_i)=\mbox{ucl}(A_{h_i})$.

Now, because $\mathbf{r}_G$ is bijective, there exists $b\in\mbox{Map}(G,K^c)$ such that
\[
\mathbf{r}_G(b)=\mathbf{r}_G(b_1)\mathbf{r}_G(b_2).
\]
By (\ref{resol1}) and Proposition~\ref{NBG} (a), in fact $b\in K_{h_1h_2}$ and $K_{h_1h_2}=KG\cdot b$. Moreover, by Proposition~\ref{SDcriteria'} (b), clearly $b$ is self-dual with respect to $Tr_{h_1h_2}$. Observe that for each $v\in M_K$, either $(h_1)_v$ or $(h_2)_v$ is unramified because $d(h_1)\cap d(h_2)=\emptyset$. So, there exists $a_v\in A_{(h_1h_2)_v}$ such that $A_{(h_1h_2)_v}=\mathcal{O}_{K_v}G\cdot a_v$ and 
\[
\mathbf{r}_G(a_v)=\mathbf{r}_G(a_{1,v})\mathbf{r}_G(a_{2,v})
\]
by \cite[Proposition 5.3 (c)]{T}. Then, for each $v\in M_K$, equation (\ref{resolvendeq}) implies that
\[
\mathbf{r}_G(a_v)=c_{1,v}c_{2,v}\cdot\mathbf{r}_G(b)
\]
and so $a_v=c_{1,v}c_{2,v}\cdot b$. As in Subsection~\ref{s:5.0}, this implies that $\mbox{ucl}(A_{h_1h_2})=j_{(s)}(c_1c_2)=\mbox{ucl}(A_{h_1})\mbox{ucl}(A_{h_2})$, which proves (b).
\end{proof}

\subsection{Summary of Main Ideas}\label{s:5.25} Let $F$ be a number field. Moreover, assume that $G$ is abelian and of odd order. To prove Theorems~\ref{thm:subgroupu} and~\ref{thm:wildu}, we will compute and characterize the  $A$-realizable classes in $\mbox{UCl}(\mathcal{O}_FG)$. Since most of the work has already been done in \cite{T}, below we will only sketch the main ideas involved.

Let $h\in\mbox{Hom}(\Omega_F,G)$ be weakly ramified. As in Subsection~\ref{s:5.0}, let $b\in F_h$ be a self-dual element such that $F_h=FG\cdot b$. Moreover, for each $v\in M_F$, let $a_v\in A_{h_v}$ be such that $A_{h_v}=\mathcal{O}_{F_v}G\cdot a_v$ and let $c_v\in (F_vG)^\times$ be such that $a_v=c_v\cdot b$. Then, we have $c:=(c_v)\in J(FG_{(s)})$ and $j_{(s)}(c)=\mbox{ucl}(A_h)$. Recall also from (\ref{resolvendeq}) that the equation $a_v=c_v\cdot b$ is equivalent to
\begin{equation}\label{eq1}
\mathbf{r}_G(a_v)=c_v\cdot \mathbf{r}_G(b).
\end{equation}
The resolvend $\mathbf{r}_G(b)$ of a self-dual element $b\in F_h$ satisfying $F_h=FG\cdot b$ is already characterized by Propositions~\ref{NBG} (a) and~\ref{SDcriteria'} (b). Hence, in order to characterize the class $\mbox{ucl}(A_h)$, it suffices to characterize the resolvend $\mathbf{r}_G(a_v)$ of an element $a_v\in A_{h_v}$ satisfying $A_{h_v}=\mathcal{O}_{F_v}G\cdot a_v$ for each $v\in M_F$. In fact, we will use \emph{reduced resolvends}, which we define in Subsection~\ref{s:5.2}.

For each $v\in M_F$, the resolvend $\mathbf{r}_G(a_v)$ of an element $a_v\in A_{h_v}$ satisfying $A_{h_v}=\mathcal{O}_{F_v}G\cdot a_v$ may be computed as follows. First of all, by \cite[Proposition 9.2]{T}, we may write $h_v=h_{v,1}h_{v,2}$ for some $h_{v,1},h_{v,2}\in\mbox{Hom}(\Omega_{F_v},G)$ such that $h_{v,1}$ is unramified and $F_v^{h_{v,2}}/F_v$ is totally ramified. Then, using \cite[Proposition 5.3 (c)]{T}, we may decompose
\begin{equation}\label{eq2}
\mathbf{r}_G(a_v)=\mathbf{r}_G(a_{v,1})\mathbf{r}_G(a_{v,2}),
\end{equation}
where $\mathcal{O}_{h_{v,1}}=\mathcal{O}_{F_v}G\cdot a_{v,1}$ and $A_{h_{v,2}}=\mathcal{O}_{F_v}G\cdot a_{v,2}$. The resolvend $\mathbf{r}_G(a_{v,1})$ of such an element $a_{v,1}$ is already characterized by Proposition~\ref{NBG} (b). On the other hand, the resolvend $\mathbf{r}_G(a_{v,2})$ may be described using the \emph{modified Stickelberger transpose} (see \cite[Propositions 10.2 and 13.2]{T}), which we define in Subsection~\ref{s:5.3}. Using results already proved in \cite{T}, we will give a complete characterization of the set
\[
\mathcal{A}^t_u(\mathcal{O}_FG):=\{\mbox{ucl}(A_h):h\in H^1_t(\Omega_F,G)\}
\]
of tame $A$-realizable classes in $\mbox{UCl}(\mathcal{O}_FG)$ in (\ref{tamecharu}). 

\subsection{Cohomology and Reduced Resolvends}\label{s:5.2}Let $F$ be a number field or a finite extension of $\mathbb{Q}_p$. Moreover, assume that $G$ is abelian.

First of all, following \cite[Sections 1 and 2]{McCulloh}, we will use cohomology to define reduced resolvends. Recall that $\Omega_F$ acts trivially on $G$ and define
\[
\mathcal{H}(FG):=((F^cG)^\times/G)^{\Omega_F}.
\]
Taking $\Omega_F$-cohomology of the exact sequence
\begin{equation}\label{exact1}
\begin{tikzcd}[column sep=1cm, row sep=1.5cm]
1 \arrow{r} &
G \arrow{r} &
(F^{c}G)^{\times} \arrow{r} &
(F^{c}G)^{\times}/G \arrow{r}&
1
\end{tikzcd}
\end{equation}
yields the exact sequence
\[
\begin{tikzcd}[column sep=1cm, row sep=1.5cm]
1 \arrow{r} &
G \arrow{r} &
(FG)^{\times} \arrow{r} &
\mathcal{H}(FG) \arrow{r}{\delta}&
\mbox{Hom}(\Omega_F,G) \arrow{r}&
1,
\end{tikzcd}
\]
where exactness on the right follows from the fact that $H^1(\Omega_F,(F^cG)^\times)=1$, which is Hilbert's Theorem 90. Alternatively, notice that a coset $\mathbf{r}_G(a)G\in\mathcal{H}(FG)$ lies in the preimage of $h\in\mbox{Hom}(\Omega_F,G)$ if and only if
\[
h(\omega)=\mathbf{r}_G(a)^{-1}(\omega\cdot\mathbf{r}_G(a))\hspace{1cm}
\mbox{for all }\omega\in\Omega_F,
\]
which is equivalent to $F_h=FG\cdot a$ by (\ref{resol2}) and Proposition~\ref{NBG} (a). By the Normal Basis Theorem, for any $h\in\mbox{Hom}(\Omega_F,G)$ there always exists $a\in F_h$ for which $F_h=FG\cdot a$. This shows that $\delta$ is indeed surjective.

The same argument as above also shows that
\begin{equation}\label{global}
\mathcal{H}(FG)=\{r_G(a)\mid F_h=FG\cdot a\mbox{ for some }h\in\mbox{Hom}(\Omega_F,G)\}.
\end{equation}
Similarly, we may define
\[
\mathcal{H}(\mathcal{O}_FG):=((\mathcal{O}_{F^c}G)^\times/G)^{\Omega_F}.
\]
Then, the argument above together with Proposition~\ref{NBG} (b) imply that
\[
\mathcal{H}(\mathcal{O}_FG)=\{r_G(a)\mid \mathcal{O}_h=\mathcal{O}_FG\cdot a\mbox{ for some $h\in\mbox{Hom}(\Omega_F,G)$ unramified}\}.
\]
In view of Proposition~\ref{SDcriteria'}, we will further define
\begin{align*}
\mathcal{H}(FG_{(s)})&:=\{r_G(a)\in\mathcal{H}(FG)\mid \mathbf{r}_G(a)\mathbf{r}_G(a)^{[-1]}\in(\mathcal{O}_FG)^\times\};\\
\mathcal{H}(FG_{(1)})&:=\{r_G(a)\in\mathcal{H}(FG)\mid \mathbf{r}_G(a)\mathbf{r}_G(a)^{[-1]}=1\},
\end{align*}
which are clearly subgroups of $\mathcal{H}(FG)$. Moreover, it is clear that both of the conditions $\mathbf{r}_G(a)\mathbf{r}_G(a)^{[-1]}\in(\mathcal{O}_FG)^\times$ and $\mathbf{r}_G(a)\mathbf{r}_G(a)^{[-1]}=1$ are independent of  the choice of the representative $\mathbf{r}_G(a)$. 

\begin{definition}Let $\mathbf{r}_G(a)G\in\mathcal{H}(FG)$. Define
\[
r_G(a):=\mathbf{r}_G(a)G,
\]
called the \emph{reduced resolvend of $a$}. Moreover, define $h_a\in\mbox{Hom}(\Omega_F,G)$ by
\[
h_a(\omega):=\mathbf{r}_G(a)^{-1}(\omega\cdot\mathbf{r}_G(a)),
\]
called the \emph{homomorphism associated to $r_G(a)$}. This definition is independent of the choice of the representative $\mathbf{r}_G(a)$, and we have $F_h=FG\cdot a$ by Proposition~\ref{NBG} (a) and (\ref{resol2}).
\end{definition}

\begin{definition}\label{JH}
For $F$ a number field, let $J(\mathcal{H}(FG))$ and $J(\mathcal{H}(FG_{(s)}))$ be the restricted direct products of the groups $\mathcal{H}(F_vG)$ and $\mathcal{H}(F_vG_{(s)})$, respectively, with respect to the subgroups $\mathcal{H}(\mathcal{O}_{F_v}G)$ for $v\in M_F$. Moreover, let
\[
\eta:\mathcal{H}(FG)\longrightarrow J(\mathcal{H}(FG))
\]
be the diagonal map and let
\[
U(\mathcal{H}(\mathcal{O}_FG)):=\prod_{v\in M_F}\mathcal{H}(\mathcal{O}_{F_v}G)
\]
be the group of unit id\`{e}les.
\end{definition}

Next, we explain how reduced resolvends may be interpreted as functions on characters of $G$. To that end, define $\det:\mathbb{Z}\widehat{G}\longrightarrow\widehat{G}$ by
\[
\det\left(\sum_\chi n_\chi\chi\right):=\prod_\chi\chi^{n_\chi}
\]
and set
\[
S_{\widehat{G}}:=\ker(\det).
\]
Then, applying the functor $\mbox{Hom}(-,(F^c)^\times)$ to the short exact sequence
\[
\begin{tikzcd}[column sep=1.5cm, row sep=1.5cm]
1 \arrow{r} &
S_{\widehat{G}} \arrow{r} &
\mathbb{Z}\widehat{G}\arrow{r}[font=\large]{\det} &
\widehat{G} \arrow{r}&
1
\end{tikzcd}
\]
yields the short exact sequence
\begin{equation}\label{exact2}
\begin{tikzcd}[column sep=0.45cm, row sep=1.5cm]
1 \arrow{r} &
\mbox{Hom}(\widehat{G},(F^{c})^{\times}) \arrow{r} &
\mbox{Hom}(\mathbb{Z}\widehat{G},(F^{c})^{\times}) \arrow{r}&
\mbox{Hom}(S_{\widehat{G}},(F^{c})^{\times}) \arrow{r}&
1,
\end{tikzcd}
\end{equation}
where exactness on the right follows from the fact that $(F^c)^\times$ is divisible and thus injective.

Observe that we have canonical identifications
\[
(F^{c}G)^{\times}=\mbox{Map}(\widehat{G},(F^{c})^{\times})
=\mbox{Hom}(\mathbb{Z}\widehat{G},(F^c)^\times).
\]
The second identification is given by extending the maps $\widehat{G}\longrightarrow (F^c)^\times$ via $\mathbb{Z}$-linearity, and the first is induced by characters (see \cite[(7.7) and (7.8)]{T}, for example). Since $G=\mbox{Hom}(\widehat{G},(F^c)^\times)$ canonically, the thirds terms
\[
(F^cG)^\times/G=\mbox{Hom}(S_{\widehat{G}},(F^c)^\times)
\]
in (\ref{exact1}) and (\ref{exact2}), respectively, are naturally identified as well. Taking $\Omega_F$-invariants, we then obtain the identification
\begin{equation}\label{iden}
\mathcal{H}(FG)=\mbox{Hom}_{\Omega_F}(S_{\widehat{G}},(F^c)^\times).
\end{equation}
Under this identification, we have
\begin{equation}\label{integral}
\mathcal{H}(\mathcal{O}_FG)\subset\mbox{Hom}_{\Omega_F}(S_{\widehat{G}},\mathcal{O}_{F^c}^\times).
\end{equation}
This inclusion is an equality when $F$ is a finite extension of $\mathbb{Q}_p$, where $p$ does not divide $|G|$ (see \cite[Proposition 7.4]{T}, for example). 

Finally, we will define
\[
rag_F:(FG)^\times\longrightarrow\mathcal{H}(FG)
\]
to be the homomorphism induced by the quotient map $(FG)^\times\longrightarrow(FG)^\times/G$.

\begin{definition}\label{rag} For $F$ a number field, observe that the homomorphism
\begin{equation}\label{rag1}
\prod_{v\in M_F}rag_{F_v}:J(FG)\longrightarrow J(\mathcal{H}(FG))
\end{equation}
is clearly well-defined, and that the diagram
\[
\begin{tikzpicture}[baseline=(current bounding box.center)]
\node at (0,2.5) [name=13] {$(FG)^\times$};
\node at (5.5,2.5) [name=14] {$J(FG)$};
\node at (0,0) [name=23] {$\mathcal{H}(FG)$};
\node at (5.5,0) [name=24] {$J(\mathcal{H}(FG))$};
\path[->,font=\large]
(13) edge node[auto]{$\partial$} (14)
(23) edge node[below]{$\eta$} (24)
(14) edge node[right]{$\prod_v rag_{F_v}$} (24)
(13) edge node[left]{$rag_F$} (23);
\end{tikzpicture}
\]
commutes. By abuse of notation, we will denote the map in (\ref{rag1}) by $rag=rag_F$.
\end{definition}

\subsection{The Modified Stickelberger Transpose}\label{s:5.3}Let $F$ be a number \mbox{field or} a finite extension of $\mathbb{Q}_p$. Moreover, assume that $G$ is abelian and of odd order. 

We will recall the definition of the modified Stickelberger transpose, which was introduced by the author in \cite[Section 8]{T}. Recall that we chose a compatible set $\{\zeta_n:n\in\mathbb{Z}^+\}$ of primitive roots of unity in $F^c$.

\begin{definition}\label{Stickel}For each $\chi\in\widehat{G}$ and $s\in G$, let
\[
\upsilon(\chi,s)\in \left[\frac{1-|s|}{2},\frac{|s|-1}{2}\right]
\]
be the unique integer (recall that $G$ has odd order) such that
\[
\chi(s)=(\zeta_{|s|})^{\upsilon(\chi,s)}
\]
and define
\[
\langle\chi,s\rangle_{*}:=\upsilon(\chi,s)/|s|.
\]
Extending this definition by $\mathbb{Q}$-linearity, we obtain a pairing
\[
\langle\hspace{1mm},\hspace{1mm}\rangle_*:\mathbb{Q}\widehat{G}\times\mathbb{Q}G\longrightarrow\mathbb{Q},
\]
called the \emph{modified Stickelberger pairing}. The map
\[
\Theta_{*}:\mathbb{Q}\widehat{G}\longrightarrow\mathbb{Q}G;
\hspace{1em}
\Theta_{*}(\psi):=\sum_{s\in G}\langle\psi,s\rangle_{*}s
\]
is called the \emph{modified Stickelberger map}.
\end{definition}

\begin{prop}\label{A-ZG}For $\psi\in\mathbb{Z}\widehat{G}$, we have $\Theta_{*}(\psi)\in\mathbb{Z}G$ if and only if $\psi\in S_{\widehat{G}}$.
\end{prop}
\begin{proof}See \cite[Proposition 8.2]{T}.
\end{proof}

Up until now, we have let $\Omega_F$ act trivially on $G$. Below, we introduce other $\Omega_F$-actions on $G$, one of which will make the $\mathbb{Q}$-linear map $\Theta_*:\mathbb{Q}\widehat{G}\longrightarrow\mathbb{Q}G$ preserve $\Omega_F$-action. Here, the $\Omega_F$-action on $\widehat{G}$ is the canonical one \mbox{induced by} the $\Omega_F$-action on the roots of unity.

\begin{definition}\label{cyclotomic}
Let $m=\exp(G)$ and let $\mu_m$ be the group of $m$-th roots of unity in $F^c$. The \emph{$m$-th cyclotomic character of $\Omega_F$} is the homomorphism
\[
\kappa:\Omega_F\longrightarrow(\mathbb{Z}/m\mathbb{Z})^{\times}
\]
defined by the equations
\[
\omega(\zeta)=\zeta^{\kappa(\omega)}\hspace{1cm}\mbox{for $\omega\in\Omega_F$ and }\zeta\in\mu_m.
\]
For $n\in\mathbb{Z}$, let $G(n)$ be the group $G$ equipped with the $\Omega_F$-action given by
\[
\omega\cdot s:=s^{\kappa(\omega^{n})}\hspace{1cm}\mbox{for $s\in G$ and $\omega\in\Omega_F$}.
\]
\end{definition}

We will need $G(-1)$. But of course, if $F$ contains the $m$-th roots of unity, then $\kappa$ is trivial and $G(n)=G(0)$ is equipped with the trivial $\Omega_F$-action for all $n\in\mathbb{Z}$.

\begin{prop}\label{eqvariant}
The map $\Theta_{*}:\mathbb{Q}\widehat{G}\longrightarrow\mathbb{Q}G(-1)$ preserves $\Omega_F$-action.
\end{prop}
\begin{proof}See \cite[Proposition 8.4]{T}.
\end{proof}

From Propositions~\ref{A-ZG} and \ref{eqvariant}, we obtain an $\Omega_F$-equivariant map
\[
\Theta_{*}:S_{\widehat{G}}\longrightarrow\mathbb{Z}G(-1).
\]
Applying the functor $\mbox{Hom}(-,(F^c)^\times)$ then yields an $\Omega_F$-equivariant homomorphism
\[
\Theta_{*}^{t}:\mbox{Hom}(\mathbb{Z}G(-1),(F^{c})^{\times})\longrightarrow\mbox{Hom}(S_{\widehat{G}},(F^{c})^{\times});\hspace{1em}f\mapsto f\circ\Theta_*.
\]
Via restriction, we then obtain a homomorphism
\[
\Theta^t_{*}=\Theta^t_{*,F}:\mbox{Hom}_{\Omega_F}(\mathbb{Z}G(-1),(F^{c})^{\times})\longrightarrow\mbox{Hom}_{\Omega_F}(S_{\widehat{G}},(F^{c})^{\times}),
\]
called the \emph{modified Stickelberger transpose}. 

Notice that we have a natural identification
\[
\mbox{Hom}_{\Omega_F}(\mathbb{Z}G(-1),(F^c)^\times)
=\mbox{Map}_{\Omega_F}(G(-1),(F^c)^\times).
\]
To simplify notation, let
\begin{align*}
\Lambda(FG)&:=\mbox{Map}_{\Omega_F}(G(-1),F^c);\\
\Lambda(\mathcal{O}_FG)&:=\mbox{Map}_{\Omega_F}(G(-1),\mathcal{O}_{F^c}).
\end{align*}
Then, we may view $\Theta_{*}^{t}$ as a homomorphism
\[
\Theta^t_*:\Lambda(FG)^{\times}\longrightarrow\mathcal{H}(FG)
\]
(recall the identification in (\ref{iden})).

\begin{prop}\label{ImTheta}We have $\Theta^t_*(\Lambda(FG)^\times)\subset\mathcal{H}(FG_{(1)})$.
\end{prop}
\begin{proof}See \cite[Proposition 8.5]{T}.
\end{proof}

\begin{definition}\label{Theta}
For $F$ a number field, let $J(\Lambda(FG))$ be the restricted direct product of the groups $\Lambda(F_vG)^\times$ with respect to the subgroups $\Lambda(\mathcal{O}_{F_v}G)^\times$ for $v\in M_F$. Moreover, let
\[
\lambda:\Lambda(FG)^\times\longrightarrow J(\Lambda(FG))
\]
be the diagonal map and let
\[
U(\Lambda(\mathcal{O}_FG)):=\prod_{v\in M_F}\Lambda(\mathcal{O}_{F_v}G)^\times
\]
be the group of unit id\`{e}les.

Next, observe that the homomorphism
\begin{equation}\label{Theta1}
\prod_{v\in M_F}\Theta^t_{*,F_v}:J(\Lambda(FG))\longrightarrow J(\mathcal{H}(FG))
\end{equation}
is well-defined since the inclusion (\ref{integral}) is an equality for all but finitely many $v\in M_F$. Because we chose $\{i_v(\zeta_n):n\in\mathbb{Z}^+\}$ to be the compatible set of primitive roots of unity in $F_v^c$, the diagram
\begin{equation}\label{Theta'}
\begin{tikzpicture}[baseline=(current bounding box.center)]
\node at (0,2.5) [name=13] {$\Lambda(FG)^\times$};
\node at (5.5,2.5) [name=14] {$J(\Lambda(FG))$};
\node at (0,0) [name=23] {$\mathcal{H}(FG)$};
\node at (5.5,0) [name=24] {$J(\mathcal{H}(FG))$};
\path[->, font=\large]
(13) edge node[auto]{$\lambda$} (14)
(23) edge node[below]{$\eta$} (24)
(14) edge node[right]{$\prod_v\Theta^t_{*,F_v}$} (24)
(13) edge node[left]{$\Theta^t_{*,F}$} (23);
\end{tikzpicture}
\end{equation}
commutes. By abuse of notation, we will denote the map in (\ref{Theta1}) by $\Theta^t_*=\Theta^t_{*,F}$.
\end{definition}

\subsection{Approximation Theorems} \label{s:5.4} Let $F$ be a number field. Moreover, assume that $G$ is abelian and of odd order.

First, we will give a preliminary characterization of the set $\mathcal{A}^t_u(\mathcal{O}_FG)$ using results already proved in \cite{T}. To that end, we need one further definition (see \cite[Section 10]{T} for the motivation of the definition).

\begin{definition}\label{primeF}For each $v\in M_F$, choose a uniformizer $\pi_{F_v}$ in $F_v$ and let $q_{F_v}$ be the order of the residue field $\mathcal{O}_{F_v}/(\pi_{F_v})$. For each $s\in G$ of order dividing $q_{F_v}-1$, define
\[
f_{F_v,s}\in\Lambda(F_vG)^\times;
\hspace{1em}f_{F_v,s}(t):=\begin{cases}
\pi_{F_v} & \mbox{if }t=s\neq1\\
1 & \mbox{otherwise}.
\end{cases}
\]
Notice that $f_{F_v,s}$ indeed preserves $\Omega_{F_v}$-action because all $(q_{F_v}-1)$-st roots of unity are contained in $F_v$, whence elements in $G$ of order dividing $q_{F_v}-1$ are fixed by $\Omega_{F_v}$, as is $\pi_{F_v}$. Let $\mathfrak{F}_{F_v}$ be the set of all such $f_{F_v,s}$ and define
\[
\mathfrak{F}=\mathfrak{F}_F:=\{f\in J(\Lambda(FG))\mid f_v\in\mathfrak{F}_{F_v}\mbox{ for all }v\in M_F\}.
\]
\end{definition}

Analogous to \cite[Theorem 11.2]{T}, we have the following theorem.

\begin{thm}\label{char1}Let $h\in\mbox{Hom}(\Omega_F,G)$ and $F_h=FG\cdot b$ with $b$ self-dual. Then, we have $h$ is tame if and only if there exists $c\in J(FG_{(s)})$ such that
\begin{equation}\label{char1'}
rag(c)=\eta(r_G(b))^{-1}u\Theta^t_*(f)
\end{equation}
for some $u\in U(\mathcal{H}(\mathcal{O}_FG))$ and $f\in\mathfrak{F}$. Moreover, if (\ref{char1'}) holds, then
\begin{enumerate}[(1)]
\item $s_v\in h_v(\Omega_{F_v})$, where $f_v=f_{F_v,s_v}$, for all $v\in M_F$;
\item $f_v=1$ if and only if $h_v$ is unramified for all $v\in M_F$;
\item $j_{(s)}(c)=\mbox{ucl}(A_h)$.
\end{enumerate}
\end{thm}
\begin{proof}First, assume that $h$ is tame. Then, by \cite[Theorem 11.2]{T}, there exists $c\in J(FG)$ such that (\ref{char1'}) holds for some $u\in U(\mathcal{H}(\mathcal{O}_FG))$ and $f\in\mathfrak{F}$. It is clear that $U(\mathcal{H}(\mathcal{O}_FG))\subset J(\mathcal{H}(FG_{(s)})$, and we have $\Theta^t_*(\mathfrak{F})\subset J(\mathcal{H}(FG_{(s)})$ by Proposition~\ref{ImTheta}. Since $b$ is self-dual, it follows from Proposition~\ref{SDcriteria'} that in fact $c\in J(FG_{(s)})$, which proves the claim.

Conversely, assume that there exists $c\in J(FG_{(s)})$ such that  (\ref{char1'}) holds for some $u\in U(\mathcal{H}(\mathcal{O}_FG))$ and $f\in\mathfrak{F}$. By \cite[Theorem 11.2]{T}, this implies that $h$ is tame, and that (1) and (2) hold. To show that (3) holds as well, notice that for each $v\in M_F$, there exists $a_v\in A_{h_v}$ such that $A_{h_v}=\mathcal{O}_{F_v}G\cdot a_v$ and
\[
r_G(a_v)=u_v\Theta^t_*(f_{v})
\]
by \cite[Theorem 10.4]{T}. In particular, we have $r_G(a_v)=rag(c_v)r_G(b)$, so there exists $t_v\in G$ such that
\[
\mathbf{r}_G(a_v)=c_v\mathbf{r}_G(b)t_v=c_vt_v\cdot\mathbf{r}_G(b).
\]
This implies that $a_v=(c_vt_v)\cdot b$. Now, set $t:=(t_v)\in U(\mathcal{O}_FG)$. Since $b$ is self-dual, as in Subsection~\ref{s:5.0}, we deduce that $\mbox{ucl}(A_h)=j_{(s)}(ct)=j_{(s)}(c)$. This proves (3) and completes the proof of the theorem.
\end{proof}

\begin{remark}The decomposition of $rag(c)$ given by (\ref{char1'}) in Theorem~\ref{char1} comes from equation (\ref{eq1}) and  decomposition (\ref{eq2}) in Subsection~\ref{s:5.25}.
\end{remark}

Theorem~\ref{char1} implies that for $c\in J(FG_{(s)})$, we have $j_{(s)}(c)\in\mathcal{A}^t_u(\mathcal{O}_FG)$ if and only if $rag(c)$ is an element of 
\begin{equation}\label{prechar}
\eta(\mathcal{H}(FG_{(1)}))U(\mathcal{H}(\mathcal{O}_FG))\Theta^t_*(\mathfrak{F})
\end{equation}
(recall Proposition~\ref{SDcriteria} (b)). However, it is unclear whether (\ref{prechar}) is a subgroup of $J(\mathcal{H}(FG))$ because $\mathfrak{F}$ is only a subset of $J(\Lambda(FG))$. Below, we state two approximation theorems. They will allow us to replace $\mathfrak{F}$ by $J(\Lambda(FG))$ in (\ref{prechar}), which will in turn allow us to prove Theorems~\ref{thm:subgroupu} and~\ref{thm:wildu}.

First, we need some further definitions.

\begin{definition}
Let $\mathfrak{m}$ be an ideal in $\mathcal{O}_F$. For each $v\in M_F$, let
\begin{align*}
U_{\mathfrak{m}}(\mathcal{O}_{F_v^c})&:=(1+\mathfrak{m}\mathcal{O}_{F_v^c})\cap(\mathcal{O}_{F_v^c})^{\times};\\
U'_{\mathfrak{m}}(\Lambda(\mathcal{O}_{F_v}G))&:=\{g_v\in\Lambda(\mathcal{O}_{F_v}G)^\times\mid g_v(s)\in U_\mathfrak{m}(\mathcal{O}_{F_v^c})\mbox{ for all }s\in G\mbox{ with }s\neq 1\},
\end{align*}
Moreover, set
\[
U'_\mathfrak{m}(\Lambda(\mathcal{O}_FG))
:=\Bigg(\prod_{v\in M_F}U_\mathfrak{m}'(\Lambda(\mathcal{O}_{F_v}G))\Bigg)\cap J(\Lambda(FG)).
\]
The \emph{modified ray class group mod $\mathfrak{m}$ of $\Lambda(FG)$} is defined by
\[
\mbox{Cl}'_{\mathfrak{m}}(\Lambda(FG)):=\frac{J(\Lambda(FG))}{\lambda(\Lambda(FG)^{\times})U'_{\mathfrak{m}}(\Lambda(\mathcal{O}_FG))}.
\]
\end{definition}

\begin{definition}\label{gs}
For $g\in J(\Lambda(FG))$ and $s\in G$, define
\[
g_s:=\prod_{v\in M_F}g_v(s)\in\prod_{v\in M_F}(F_v^c)^\times.
\]
\end{definition}

We can now state the approximation theorems.

\begin{thm}\label{approx1}Let $\mathfrak{m}$ be an ideal in $\mathcal{O}_F$ divisible by both $|G|$ and $\exp(G)^2$. Then, we have
\[
\Theta_*^t(U_\mathfrak{m}'(\Lambda(\mathcal{O}_FG))\subset U(\mathcal{H}(\mathcal{O}_FG)).
\]
\end{thm}
\begin{proof}See \cite[Theorem 11.5 (b)]{T}.
\end{proof}

\begin{thm}\label{approx2}Let $g\in J(\Lambda(FG))$ and let $T$ be a finite subset of $M_F$. Then, \mbox{there exists} $f\in\mathfrak{F}$ such that $f_v=1$ for all $v\in T$ and 
\[
g\equiv f\hspace{1cm}(\mbox{mod }\lambda(\Lambda(FG)^\times)U_\mathfrak{m}'(\Lambda(\mathcal{O}_FG))).
\]
Moreover, we may choose $f$ so that for each $s\in G(-1)$ with $s\neq 1$, there exists $\omega\in\Omega_F$ such that $f_{\omega\cdot s}\neq 1$.
\end{thm}
\begin{proof}See \cite[Proposition 6.14]{McCulloh}.
\end{proof}

\subsection{Proof of Theorems~\ref{thm:subgroupu} and~\ref{thm:wildu}}\label{s:5.5}

\begin{proof}[Proof of Theorem~\ref{thm:subgroupu}]Let $\upvarrho_u$ be the composition of the homomorphism 
\[
J(KG_{(s)})\longrightarrow J(\mathcal{H}(KG_{(s)}));\hspace{1cm}c\mapsto rag(c),
\]
where $rag$ is as in Definition~\ref{rag}, followed by the natural quotient map
\[
J(\mathcal{H}(KG_{(s)}))\longrightarrow\frac{J(\mathcal{H}(KG_{(s)}))}{\eta(\mathcal{H}(KG_{(1)}))U(\mathcal{H}(\mathcal{O}_KG))\Theta^t_*(J(\Lambda(KG)))}.
\]
Note that $\Theta^t_*(J(\Lambda(KG)))\subset J(\mathcal{H}(KG_{(s)}))$ by Proposition~\ref{ImTheta}. We will show that $\mathcal{A}_u^t(\mathcal{O}_KG)$ is a subgroup of $\mbox{UCl}(\mathcal{O}_KG)$ by showing that
\[
j_{(s)}^{-1}(\mathcal{A}_u^t(\mathcal{O}_KG))=\ker(\upvarrho_u),
\]
or equivalently, that for $c\in J(KG_{(s)})$, we have $j_{(s)}(c)\in\mathcal{A}^t_u(\mathcal{O}_KG)$ if and only if
\begin{equation}\label{tamecharu}
rag(c)\in\eta(\mathcal{H}(KG_{(1)}))U(\mathcal{H}(\mathcal{O}_KG))\Theta^t_*(J(\Lambda(KG))).
\end{equation}

To that end, let $c\in J(KG_{(s)})$ be given. If $j_{(s)}(c)=\mbox{ucl}(A_h)$ \mbox{for some tame} $h\in\mbox{Hom}(\Omega_K,G)$, with $K_h=KG\cdot b$ and $b$ self-dual say, \mbox{then $r_G(b)\in\mathcal{H}(KG_{(1)})$} by (\ref{global}) and Proposition~\ref{SDcriteria'} (b). Moreover, \mbox{by Theorem~\ref{char1}, there} exists $c'\in J(KG_{(s)})$ such that $j_{(s)}(c')=\mbox{ucl}(A_h)$ and
\[
rag(c')\in\eta(\mathcal{H}(KG_{(1)}))U(\mathcal{H}(\mathcal{O}_KG))\Theta^t_*(J(\Lambda(KG))).
\]
Since $j_{(s)}(c)=\mbox{ucl}(A_h)$ also, we have
\[
c\equiv c'\hspace{1cm}(\mbox{mod }\partial(KG_{(1)})U(\mathcal{O}_KG)).
\]
It is then clear that (\ref{tamecharu}) indeed holds.

Conversely, if (\ref{tamecharu}) holds, then
\begin{equation}\label{tamecharu'}
rag(c)=\eta(r_G(b))^{-1}u\Theta^t_*(g)
\end{equation}
for some  $r_G(b)\in\mathcal{H}(KG_{(1)})$, $u\in U(\mathcal{H}(\mathcal{O}_KG))$, and $g\in J(\Lambda(KG))$. Let $\mathfrak{m}$ be an ideal in $\mathcal{O}_K$. Then, by Theorem~\ref{approx2}, there exists $f\in\mathfrak{F}$ such that
\begin{equation}\label{g=f}
g\equiv f\hspace{1cm}(\mbox{mod }\lambda(\Lambda(KG)^\times)U_\mathfrak{m}'(\Lambda(\mathcal{O}_KG))).
\end{equation}
Choosing $\mathfrak{m}$ to be divisible by $|G|$ and $\exp(G)^2$, by Proposition~\ref{ImTheta} and Theorem~\ref{approx1}, the above implies that
\[
\Theta^t_*(g)\equiv \Theta^t_*(f)\hspace{1cm}(\mbox{mod }\eta(\mathcal{H}(KG_{(1)}))U(\mathcal{H}(\mathcal{O}_KG))).
\]
Hence, changing $b$ and $u$ in (\ref{tamecharu'}) if necessary, we may assume that $g=f$. Since $b$ is self-dual, if $h:=h_b$ is the homomorphism associated to $r_G(b)$, then $h$ is tame and $j_{(s)}(c)=\mbox{ucl}(A_h)$ by Theorem~\ref{char1}. It remains to show that $h$ may be chosen such that (1) and (2) are satisfied.

Let $T$ be a finite set of primes in $\mathcal{O}_K$. First of all, by Theorem~\ref{approx2}, we may choose the $f\in\mathfrak{F}$ in (\ref{g=f}) such that $f_v=1$ for all $v\in T$. It then follows from Theorem~\ref{char1} that $h_v$ is unramified for all $v\in T$, so (2) holds. Moreover, we may also choose the $f\in\mathfrak{F}$ in (\ref{g=f}) such that for each $s\in G(-1)$ with $s\neq 1$, there exists $\omega\in\Omega_K$ such that $f_{\omega\cdot s}\neq 1$. In particular, we have $f_v=f_{K_v,\omega\cdot s}$ for some $v\in M_K$. But observe that $\langle s\rangle=\langle\omega\cdot s\rangle$ and that $\omega\cdot s\in h_v(\Omega_{K_v})$ by Theorem~\ref{char1}, so $s\in h(\Omega_K)$. This shows that $h$ is surjective and hence $K_h$ is a field, so (1) holds as well. This completes the proof of the theorem.
\end{proof}

\begin{proof}[Proof of Theorem~\ref{thm:wildu}]Let $h\in\mbox{Hom}(\Omega_K,G)$ be given as in the statement of the theorem. Let $b\in K_h$ be as in (\ref{b}), where we choose $b$ to be self-dual. Moreover, for each $v\in M_K$, let $a_v\in A_{h_v}$ and $c_v\in (K_vG)^\times$ be given as in (\ref{av}) and (\ref{cv}), respectively. Then, we have $c\in J(KG_{(s)})$ and $\mbox{ucl}(A_h)=j_{(s)}(c)$, as explained in Subsection~\ref{s:5.0}. Moreover, equation (\ref{resolvendeq}) implies that
\[
rag(c_v)=r_G(b)^{-1}r_G(a_v)
\]
for each $v\in M_K$. Notice also that $r_G(b)\in\mathcal{H}(KG_{(1)})$ by (\ref{global}) and Proposition~\ref{SDcriteria'} (b). From (\ref{tamecharu}), we then see that $\mbox{ucl}(A_h)\in\mathcal{A}^t_u(\mathcal{O}_KG)$ will hold provided that for all $v\in M_K$, we have
\begin{equation}\label{wildu1}
r_G(a_v)\in\mathcal{H}(\mathcal{O}_{K_v}G)\Theta^t_*(\Lambda(K_vG)^\times)
\end{equation}
If $v\notin V$, then (\ref{wildu1}) follows from \cite[Theorem 10.3]{T}. If $v\in V$, then hypotheses (1) and (2) allow us to apply \cite[Theorem 13.2]{T} and conclude that (\ref{wildu1}) holds. This proves the theorem.
\end{proof}

\section{Acknowledgments}

I would like to thank my advisor Professor Adebisi Agboola for bringing this problem to my attention.


\end{document}